\documentclass[5pt]{article}
\usepackage{amsmath}
\usepackage[latin1]{inputenc}
\usepackage{amsmath,amsthm,amssymb}
\usepackage{graphicx}
\usepackage[all,poly,knot]{xy}

\usepackage{color}

\renewcommand{\theequation}{\thesection.\arabic{equation}}

\newtheorem{lem}{Lemma}[section]
\newtheorem{thm}{Theorem} [section]
\newtheorem{prop}{Proposition} [section]

\newtheorem{coro}{Corollary}[section]

\title{ Local Method for  Compositional Inverses of Permutational Polynomials
\thanks{Supported By NSF of China No. 12171163 }}

\author{Pingzhi Yuan\thanks{ P. Yuan is with School of  of Mathematical Science, South China Normal University,  Guangzhou 510631, China (email: yuanpz@scnu.edu.cn).}}

    \date{}
\begin{document}
\baselineskip15pt \maketitle
\renewcommand{\theequation}{\arabic{section}.\arabic{equation}}
\catcode`@=11 \@addtoreset{equation}{section} \catcode`@=12

    \begin{abstract}In this paper, we   provide a   local method to find  compositional inverses of all PPs, some new PPs and their compositional inverses are given.

\end{abstract}


\section{Introduction}

\,\,\, Let $q$ be a prime power, $\mathbb{F}_q$ be the finite field of order $q$, and $\mathbb{F}_q[x]$
be the ring of polynomials in a single indeterminate $x$ over $\mathbb{F}_q$. A polynomial
$f \in\mathbb{F}_q[x]$ is called a {\em permutation polynomial} (PP for short) of $\mathbb{F}_q$ if it induces
a bijective map from $\mathbb{F}_q$ to itself.

Constructing PPs and explicitly determining the compositional inverse of a PP are useful because a PP and its inverse are required in many applications. For example, involutions are particularly useful ( as a part of a block ciper) in devices with limited resources. However, it is difficult to determine whether  a given polynomial over a finite field is a PP or not, and it is not easy to find the explicit compositional inverse of a random PP, except for several well-known classes of PPs, which have very nice structure. See   \cite{CH02, DY06, LC07, LQW19, LN97, NLQW21, TW14, TW17, W07, W17, Wu14, WL13, WL13J, YD11, YD14, ZYLHZ19, ZY18, ZWW20} for more details.

In 2011, Akbrary, Ghioca and Wang  \cite{AGW11} proposed a powerful method called the AGW criterion for constructing PPs. In the present paper, we will present a local method to find the compositional inverses of PPs.

The rest of this paper is organized as follows. In Section 2, we present the local criteria for a polynomial to be a PP. In Section 3, we present some known PPs and their compositional inverses by applying the results in Section 2.  In Section 4, we use the  method developed in Section 2 to give necessary and sufficient conditions for two kinds of polynomials to be PPs, and we also give their compositional inverses.

We know that every map from $\mathbb{F}_q$ to $\mathbb{F}_q$ can be viewed as a polynomial in the ring $\left(\mathbb{F}_q[x]/(x^q-x), +, \circ\right)$, where the two operations are the addition and the composition of two polynomials modulo $x^q-x$. For a polynomial, $f(x)$ is a PP over $\mathbb{F}_q$ if and only if $f(x)$ is an unit in the ring $\left(\mathbb{F}_q[x]/(x^q-x), +, \circ\right)$, and we use $f^{-1}(x)$ to denote the compositional inverse of $f(x)$. Hence we view the polynomials over $\mathbb{F}_q$ as  elements in the ring $\left(\mathbb{F}_q[x]/(x^q-x), +, \circ\right)$ thorough the paper.

\section{Local criteria for a polynomial to be a PP}

To begin with, we give the following simple and useful result for a map to be a bijection. We have

% $$\xymatrix{
 % A \ar[d]_{\varphi} \ar[r]^{f}
  %              & A \ar[d]^{\psi}  \\
%  S  \ar[r]_{h}
%                & S          }$$

\begin{lem}\label{lem2.1}   {\rm ( Local criterion) }
Let $A$ and   $S$ be finite sets and let $f:A\rightarrow A$ be a map. Then $f$ is a bijection if and only if for any surjection $\psi: A\rightarrow S$, $\varphi=\psi\circ f$ is a surjection and $f$ is
injective on $\varphi^{-1}(s)$ for each $s\in S$.
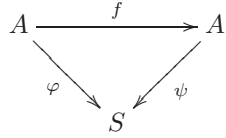
$$\xymatrix{
  A \ar[rr]^{f} \ar[dr]_{\varphi}
                &  &   A \ar[dl]^{\psi}    \\
                & S                 }$$
\end{lem}
\begin{proof}The necessity is obvious.

Now we prove the sufficiency. If $f(a)=f(b)$ for some $a, b\in A$, then
$$\psi(f(a))=\psi(f(b)).$$
That is $\varphi(a)=\varphi(b)=s, s\in S$. Hence
$$a, b\in\varphi^{-1}(s).$$ Since  $f$ is
injective on $\varphi^{-1}(s)$ for each $s\in S$ and $f(a)=f(b)$, we get $a=b$. Therefore $f$ is a bijection.\end{proof}

Applying Lemma \ref{lem2.1}, we get the following result, which is called AGW criterion now.

\begin{lem}\label{lem2.2}   {\rm ( AGW criterion) }
Let $A, S$ and $\bar{S}$ be finite sets with $\sharp
S=\sharp\bar{S}$, and let $f:A\rightarrow A, h:
S\rightarrow\bar{S}$, $\lambda: A\rightarrow S$, and $\bar{\lambda}:
A\rightarrow \bar{S}$ be maps such that $\bar{\lambda}\circ
f=h\circ\lambda$. If both $\lambda$ and $\bar{\lambda}$ are
surjective, then the following statements are equivalent:

(i) $f$ is a bijective (a permutation of $A$); and

(ii) $h$ is a bijective from $S$ to $\bar{S}$ and if $f$ is
injective on $\lambda^{-1}(s)$ for each $s\in S$.
\end{lem}
\begin{proof} Applying Lemma \ref{lem2.1} to $\psi=\bar{\lambda}$ and $\varphi=\bar{\lambda}\circ
f=h\circ\lambda$. Since $\lambda$ is a surjection and $\sharp
S=\sharp\bar{S}$, we have $\varphi=h\circ\lambda$ is a surjection if and only if $h$ is a bijection. On the other hand, if $h$ is a bijection, then we have
$$\{\lambda^{-1}(s), s\in S\}=\{\varphi^{-1}(\bar{s}), \bar{s}\in \bar{S}\}.$$
Hence we have the result.\end{proof}

We also have the following corollary.
\begin{coro}Let $A, S$ and $\bar{S}$ be finite sets with $\sharp
S=\sharp\bar{S}$, and let $f:A\rightarrow A, h:
S\rightarrow\bar{S}$, $\lambda: A\rightarrow S$, and $\bar{\lambda}:
A\rightarrow \bar{S}$ be maps such that $\bar{\lambda}\circ
f=h\circ\lambda$. If  $\bar{\lambda}$ is a
surjection and $h$ is a bijection, then the following statements are equivalent:

(i) $f$ is a bijective (a permutation of $A$); and

(ii) $\lambda$ is a surjection and $f$ is
injective on $\lambda^{-1}(s)$ for each $s\in S$.
\end{coro}

We now prove the following general result on PP.
\begin{thm}\label{th2.1} A polynomial $f(x)\in\mathbb{F}_q[x]$ is a PP if and only if for any maps $\psi_i, \, i=1, \ldots, t, t\in\mathbb{N}$ (we also denote them as $\psi_i(x)\in\mathbb{F}_q[x]/(x^q-x)$) such that
$F(\psi_1(x), \ldots, \psi_t(x))=x$ for some polynomial $F(x_1, \ldots, x_t)\in \mathbb{F}_q[x_1, \ldots, x_t]$, there exists a polynomial $G(x_1, \ldots, x_t)\in \mathbb{F}_q[x_1, \ldots, x_t]$ satisfies $G(\psi_1(f(x)), \ldots, \psi_t(f(x)))=x$. Moreover, if $f(x)$ is a PP, then
$$f^{-1}(x)=G(\psi_1(x), \ldots, \psi_t(x)),$$
where $f^{-1}(x)$ denotes  the compositional inverse of $f(x)$.\end{thm}
\begin{proof} We first prove the necessity. Assume that $f(x)$ is a PP, let
$$\varphi_i(x)=\psi_i(f(x)), \, i=1, \ldots, t,$$
then we have
$$\psi_i(x)= \varphi_i (f^{-1}(x)), \, i=1, \ldots, t,$$
where $f^{-1}(x)$ denotes the compositional inverse of $f(x)$.   Since $F(\psi_1(x), \ldots, \psi_t(x))=x$, we get
$$F(\varphi_1(f^{-1}(x)), \ldots, \varphi_t(f^{-1}(x)))=x.$$
It follows that
$$F(\varphi_1(x), \ldots, \varphi_t(x))=f(x)\,\, \mbox{and}\,\, f^{-1}\left(F(\varphi_1(x), \ldots, \varphi_t(x))\right)=x.$$
Let $G(x_1, \ldots, x_t)=f^{-1}\circ F(x_1, \ldots, x_t)$, then we have $G(\varphi_1(x), \ldots, \varphi_t(x))=x$, that is
$$G(\psi_1(f(x)), \ldots, \psi_t(f(x)))=x.$$

Next, we prove the sufficiency. Since for any maps $\psi_i, \, i=1, \ldots, t, t\in\mathbb{N}$ such that
$F(\psi_1(x), \ldots, \psi_t(x))=x$ for some polynomial $F(x_1, \ldots, x_t)\in \mathbb{F}_q[x_1, \ldots, x_t]$, there exists a polynomial $G(x_1, \ldots, x_t)\in \mathbb{F}_q[x_1, \ldots, x_t]$ satisfies $G(\psi_1(f(x)), \ldots, \psi_t(f(x)))=x$. Let
$$g(x)=G(\psi_1(x), \ldots, \psi_t(x)),$$
then $g\circ f(x)=x$, which implies that $g(x)=f^{-1}(x)$, i.e., $f(x)$ is a PP. We are done. \end{proof}

We also have the following alternative result.
\begin{thm}\label{th2.2} Let $q$ be a prime power and $f(x)$ be a polynomial over $\mathbb{F}_q$. Then $f(x)$ is a PP if and only if there exist nonempty finite subsets $S_i, i=1, \ldots, t$  of $\mathbb{F}_q$ and maps $\psi_i: \mathbb{F}_q\to S_i, i=1, \ldots, t$  such that $\psi_i\circ f=\varphi_i, i=1, \ldots, t$ and $x=F(\varphi_1(x), \ldots, \varphi_t(x))$, where $F(x_1, \ldots, x_t)\in\mathbb{F}_q[x_1, \ldots, x_t]$. Moreover, we have
$$ f^{-1}(x)=F(\psi_1(x), \ldots, \psi_t(x)).$$\end{thm}
\begin{proof} If there exist nonempty finite subsets $S_i, i=1, \ldots, t$  of $\mathbb{F}_q$ and maps $\psi_i: \mathbb{F}_q\to S_i, i=1, \ldots, t$  such that $\psi_i\circ f=\varphi_i, i=1, \ldots, t$ and $x=F(\varphi_1(x), \ldots, \varphi_t(x))$, then we the following diagram commutes. $$\xymatrix{
  \mathbb{F}_q \ar[d]_{(\varphi_1, \ldots, \varphi_t)} \ar[r]^{f} & \mathbb{F}_q \ar[dr]^{(\psi_1, \ldots, \psi_t)} \\
  S_1\times\cdots\times S_t  \ar[dr]_{F(x_1, \ldots, x_t)}
 &    & S_1\times\cdots\times S_t \ar[dl]^{F(x_1, \ldots, x_t)} \\
  & \mathbb{F}_q  }                                        $$
 Hence
 $$F( \psi_1(f(x)), \ldots, \psi_t(f(x)))=x,$$
 which implies that
 $$ f^{-1}(x)=F(\psi_1(x), \ldots, \psi_t(x)).$$

 If $f(x)$ is a permutational polynomial over $\mathbb{F}_q$, we can take $S=\mathbb{F}_q$, $i=1$, $\psi=f^{-1}(x)$, $\varphi(x)=x$ and $F(x)=x$, then all conditions are satisfied,  and we have$f^{-1}(x)=\psi(x)$.
  \end{proof}

%$$\xymatrix{
%  A \ar[dr]_{\varphi}  \ar[rr]^{f}  \ar[rr]^{f^{-1}}
%                & A \ar[d]^{\varphi} & A \\
%                & S             }$$

%$$\xymatrix{
%  A \ar[d]_{\varphi} \ar[r]^{f} & A \ar[d]_{\varphi} \ar[r]^{f^{-1}} & A \ar[d]^{\varphi} \\
%  S \ar[r]^{id} & S \ar[r]^{id} & S   }$$

\section{Some old PPs and their compositional inverses}

In this section,  we give new proofs for some known PPs and their compositional inverses by applying the results in Section 2.  To do this, we need to find out  two polynomials $\psi_1(x)$ and $\psi_2(x)$ such that we can find a polynomial $G(x, y)\in \mathbb{F}_q[x, y]$ satisfying
$$x=G(\psi_1(f(x)), \psi_2(f(x))).$$
Now applying Theorem \ref{th2.2}, we get
$$f^{-1}(x)=G(\psi_1(x), \psi_2(x)).$$

We give two examples below. The following result was discovered independently by several authors, we will give a new proof of Theorem 11 in \cite{NLQW21} by applying Theorem \ref{th2.2}.
\begin{lem}\label{le10} {\rm (\cite[Theorem 2.3]{PL01} \cite[Theorem 1]{W07} \cite[Lemma 2.1]{Z09})}Let $q$ be a prime power and $f(x)=x^rh(x^s)\in\mathbb{F}_q[x]$, where $s=\frac{q-1}{\ell}$ and $\ell$ is an integer. Then $f(x)$ permutes $\mathbb{F}_q$ if and only if

(1) $\gcd(r, \, s)=1$ and

(2) $g(x)=x^rh(x)^s$ permutes $\mu_{\ell}$. \end{lem}

\begin{prop}{\rm (\cite[Theorem 11]{NLQW21})} Let $f(x)=x^rh(x^s)\in\mathbb{F}_q[x]$ defined in Lemma \ref{le10} be a permutation over $\mathbb{F}_q$ and $g^{-1}(x)$ be the compositional inverse of $g(x)=x^rh(x)^s$ over $\mu_{\ell}$. Suppose $a$ and $b$ are two integers satisfying $as+br=1$. Then the compositional inverse of $f(x)$ in $\mathbb{F}_q[x]$ is given by
$$f^{-1}(x)=g^{-1}(x^s)^ax^bh(g^{-1}(x^s))^{-b}.$$
\end{prop}
\begin{proof} By the assumptions, we have the following commutative diagrams$$ \xymatrix{
  \mathbb{F}_q \ar[rr]^{f(x)} \ar[dr]_{\varphi_i}
                &  &    \mathbb{F}_q \ar[dl]^{\psi_i}    \\
                & \mu_{\ell}                 }\quad i=1, 2,
$$
where $$\psi_1(x)=g^{-1}(x^s),  \varphi_1(x)=x^s, \quad \psi_2(x)=x, \quad \varphi_2(x)=f(x)=x^rh(x^s).$$
We have
$$\varphi_1(x)^a\cdot\left(\frac{\varphi_2(x)}{h(\varphi_1(x))}\right)^b=x,$$
it follows from Theorem \ref{th2.2} that
$$f^{-1}(x)=\psi_1(x)^a\cdot\left(\frac{\psi_2(x)}{h(\psi_1(x))}\right)^b=\left(g^{-1}(x^s)\right)^ax^b\left(h(g^{-1}(x^s))\right)^{-b}.$$ This completes the proof.\end{proof}

Next, we give another proof of the following result.

\begin{prop}\label{Thm1}{\rm \cite[Theorem 3.1]{Y22}} Let $q$ be a prime power, and $S, \bar{S}$  subsets of $\mathbb{F}_q^\ast$ with $\sharp S=\sharp\bar{S}$.
 Let  $f:\mathbb{F}_q^\ast\rightarrow \mathbb{F}_q^\ast, g:S\rightarrow\bar{S}$, $\lambda: \mathbb{F}_q^\ast\rightarrow S$, and $\bar{\lambda}:\mathbb{F}_q^\ast\rightarrow \bar{S}$ be maps such that both $\lambda$ and $\bar{\lambda}$ are surjective maps and $\bar{\lambda}\circ f=g\circ\lambda$.

Let $f_1(x)$ and  $f(x)=f_1(x)h(\lambda(x))$  are PPs over $\mathbb{F}_q^\ast$, and let $f_1^{-1}(x), f^{-1}(x)$ and $g^{-1}(x)$ be the compositional inverses of $f_1(x), f(x)$ and $g(x)$, respectively. Then we have
$$f^{-1}(x)=f_1^{-1}\left(\frac{x}{h(g^{-1}(\bar{\lambda}(x)))}\right).$$\end{prop}
\begin{proof}By the assumptions of the theorem, we have $$\xymatrix{
  \mathbb{F}_q \ar[rr]^{f} \ar[dr]_{\varphi_i}
                &  &    \mathbb{F}_q \ar[dl]^{\psi_i}    \\
                & S                 }\quad i=1, 2,$$
  where
                $$\psi_1(x)=g^{-1}(\bar{\lambda}),  \varphi_1(x)=\lambda, \quad \psi_2(x)=x, \varphi_2(x)=f(x).$$
It is easy to check that
                $$x=f_1^{-1}\left(\frac{\varphi_2(x)}{h(\varphi_1(x))}\right),$$
by Theorem \ref{th2.2}, we have$$ \quad f^{-1}(x)=f_1^{-1}\left(\frac{x}{h(g^{-1}(\bar{\lambda}(x)))}\right).$$And we are done.
\end{proof}

{\bf Remark:} We can also  obtain the compositional inverses of many known PPs by the method above. We stop to do this here.

\section{Some new PPs and their compositional inverses}

For more applications of the results in Section 2, we will present some new PPs and their compositional inverses in this section.

Let $d>1$ be a positive integer, $q$ a prime power with $q\equiv1\pmod{d}$ and $\omega$ a  $d$-th primitive root of unity over $\mathbb{F}_q$. For $0\le i\le d-1$, let
$$A_i(x)=x^{q^{d-1}}+\omega^ix^{q^{d-2}}+\cdots+\omega^{i(d-1)}x.$$
Most results in the following two lemmas are proved in \cite{Yuan22}, for completeness and self-contained, we present the full proof. We have
\begin{lem}\label{lem4.1} Let the notations be as above  and let $g(x)\in \mathbb{F}_q[x]$ be a polynomial, then we have

(i) $A_i^q(x)=\omega^iA_i(x), 0\le i\le d-1$.

(ii) For any positive integer $m$ and integers $i, j$ with $0\le i, j\le d-1$, we have
$$A_j(x)\circ A_i^m(x)=\left\{ \begin{array}{ll}
                     d\omega^{-j}A_i^m(x), & \mbox{if } j\equiv im\pmod{d}, \\
                     0, & \mbox{ otherwise.}
                    \end{array}
         \right.$$
(iii) For any positive integer $j$ with $1\le j\le d-1$, we have
$$A_j(x)\circ g(A_0(x))=0.$$
         \end{lem}
         \begin{proof} For (i), we have
$$\left(A_i(x)\right)^q=x+\omega^ix^{q^{d-1}}+\cdots+\omega^{i(d-1)}x^q=\omega^iA_i(x).$$
This proves (i). Now we prove (ii), repeating the procedure in (i), we get
$$\left(A_i(x)\right)^{q^t}=\omega^{it}A_i(x).$$
Hence
$$\left(A_i^m(x)\right)^{q^t}=\left(A_i(x)^{q^t}\right)^m=\omega^{imt}A_i^m(x).$$
Therefore
$$ A_j(x)\circ A_i^m(x)=A_i^m(x)\sum_{t=0}^{d-1}\omega^{(jt+i(d-1-t)m)}=A_i^m(x)\omega^{im(d-1)}\sum_{t=0}^{d-1}\omega^{(j-im)t}$$
$$=\left\{ \begin{array}{ll}
                     d\omega^{-j}A_i^m(x), & \mbox{if } j\equiv im\pmod{d}, \\
                     0, & \mbox{ otherwise.}
                    \end{array}
         \right.$$
(iii) Let $g(x)=b_0+b_1x+\cdots+b_tx^t, b_i\in\mathbb{F}_{q}, 0\le i\le t$, since $A_j(x)$ is a $q$-polynomial and $b_i\in\mathbb{F}_{q}, 0\le i\le t$, by (ii), we have
$$A_j(x)\circ g(A_0(x))=A_j(b_0)+\sum_{i=1}^tA_j(x)\circ A_0^i(x)$$
$$=A_j(b_0)=b_0(1+\omega^j+\cdots+\omega^{j(d-1)})=0.$$\end{proof}

\begin{lem}\label{lem4.2} Let the notations be as in Lemma \ref{lem4.1}, and let
$$B_i=\{A_i(x), \, x\in\mathbb{F}_{q^d}\},\,\, 0\le i\le d-1.$$
Then $B_i=\{0\}\cup y_i\mathbb{F}_q^\ast$, where $y_i$ is a fixed nonzero element of $B_i$, $1\le i\le d-1$ and $B_0=\mathbb{F}_q$.\end{lem}
\begin{proof} Obviously, $0\in B_i$ and $B_0=\mathbb{F}_q$. For any elements $y_1, y_2\in B_i$ with $y_1y_2\ne0$, by Lemma \ref{lem4.1} (i), we have
$$y_j^q=\omega^i y_j, \,\, j=1, 2.$$
Hence $(y_1/y_2)^q=y_1/y_2$, which implies that $y_1/y_2\in\mathbb{F}_q^\ast$, and the assertion follows. \end{proof}

\begin{thm}\label{th4.1}Let $d>1$ be a positive integer, $q$ a prime power with $q\equiv1\pmod{d}$ and $\omega$ a  $d$-th primitive root of unity over $\mathbb{F}_q$. Let
$$A_i(x)=x^{q^{d-1}}+\omega^ix^{q^{d-2}}+\cdots+\omega^{i(d-1)}x, \,\, 0\le i\le d-1.$$ Let $m_1, \ldots, m_{d-1}$ be positive integers  and $u_1, \ldots, u_{d-1}\in\mathbb{F}_q$, and let $g(x)\in \mathbb{F}_q[x]$ be a polynomial. Then the polynomial
$$f(x)=g(A_0(x))+\sum_{t=1}^{d-1}u_iA_i^{m_i}(x),$$
 is a PP over $\mathbb{F}_{q^d}$ if and only if $\{0\}\cup\{im_i, 1\le i\le d-1\}$ is a complete residue modulo $d$, $u_1\cdot\cdots\cdot u_{d-1}\in\mathbb{F}_q^\ast$, $\gcd(m_1\cdot\cdots\cdot m_{d-1}, q-1)=1$ and $g(x)$ is a PP over $\mathbb{F}_q$. Furthermore, if $f(x)$ is a PP over $\mathbb{F}_{q^d}$,  $r_i$ are positive integers with $m_ir_i\equiv1\pmod{d(q-1)}(1\le i\le d-1)$ and $g^{-1}(x)$ is the compositional inverse of $g(x)$, then
$$f^{-1}(x)=\frac{1}{d}\left(g^{-1}(A_0(x)/d)+\sum_{i=1, j\equiv im_i\pmod{d}}^{d-1}\omega^i(du_i\omega^{-j})^{-r_i}A_j(x)^{r_i}\right).$$\end{thm}

\begin{proof} We first prove the necessity. Suppose that there is a positive integer $j, 1\le j\le d-1$ such that $j\not\equiv im_i\pmod{d}$ for any $i\in\{1, \ldots, d-1\}$. By Lemma \ref{lem4.1} (ii) and (iii), we have
$$A_j(x)\circ\left(u_iA_i^{m_i}(x)\right)=0, \quad 1\le i\le d-1, \quad A_j(x)\circ g(A_0(x))=0,$$
hence $ A_j(x)\circ f(x)=0$ and $f(x)$ is not a PP over $\mathbb{F}_{q^d}$. On the other hand, if $u_i=0$ for some $i, 1\le i\le d-1$, then it is trivial that $\{0\}\cup\{tm_t, 1\le t\le d-1, t\ne i\}$ is not a complete residue modulo $d$, so $f(x)$ is not a PP over $\mathbb{F}_{q^d}$. As for $g(x)$, by (ii) and a direct computation we have
$$A_0(x)\circ (f(x))=dg(A_0(x)).$$
It follows from Lemma \ref{lem2.1} that $f(x)$ is a PP only if $dg(A_0(x))$ is a surjective map from $\mathbb{F}_{q}$ to $\mathbb{F}_{q}$, which implies that $g(x)$ is a PP over $\mathbb{F}_{q}$ since
$A_0(\mathbb{F}_{q^d})=B_0=\mathbb{F}_q$.

Now we assume that $u_1\cdot\cdots\cdot u_{d-1}\in\mathbb{F}_q^\ast$, $\{0\}\cup\{im_i, 1\le i\le d-1\}$ is a complete residue modulo $d$ and $g(x)$ is a PP over $\mathbb{F}_q$. Then
$$A_j(x)\circ f(x)=du_i\omega^{-j}A_i^{m_i}(x),\,\, im_i\equiv j\pmod{d}, \, 1\le j\le d-1,$$ so we have the following commutative diagram
$$\xymatrix{
  \mathbb{F}_{q^d} \ar[d]_{A_i(x)} \ar[r]^{f}
                & \mathbb{F}_{q^d} \ar[d]^{A_j(x)}  \\
  B_i  \ar[r]_{du_i\omega^{-j}x^{m_i}}
                & B_j          } $$
for each $j, 1\le j\le d-1$, where $i$ is the integer with $im_i\equiv j\pmod{d}$, $B_i=\{A_i(x), x\in \mathbb{F}_{q^d}\}, 1\le i\le d-1$. By Lemma \ref{lem4.2}, the map
$$h_{ij}: B_i\to B_j, \,\, a\mapsto du_i\omega^{-j}a^{m_i},$$
is bijective if and only if $\gcd(m_i, q-1)=1$. If $\gcd(m_i, q-1)=1$, then the compositional inverse of $du_i\omega^{-j}x^{m_i}$ is $(du_i\omega^{-j})^{-r_i}x^{r_i}$, where $m_ir_i\equiv1\pmod{d(q-1)}$. Hence we have the following commutative diagrams
$$\xymatrix{ \mathbb{F}_{q^d} \ar[rr]^{f(x)} \ar[dr]_{\varphi_i}
                &  &    \mathbb{F}_{q^d} \ar[dl]^{\psi_i}    \\
                & B_i                }\quad i=0, 1, \ldots, d-1,
$$
where $\varphi_i(x)=A_i(x), i=0, 1, \ldots, d-1$, $\psi_0(x)=g^{-1}(A_0(x)/d)$ and $\psi_i(x)=(du_i\omega^{-j})^{-r_i}A_j^{r_i}(x), i=1, \ldots, d-1$.

Note that
$$x=\frac{1}{d}\sum_{i=0}^{d-1}\omega^i A_i(x)=\frac{1}{d}\sum_{i=0}^{d-1}\omega^i\varphi_i(x),$$
by Theorem \ref{th2.2}, we have
$$f^{-1}(x)=\frac{1}{d}\left(g^{-1}(A_0(x)/d)+\sum_{i=1, j\equiv im_i\pmod{d}}^{d-1}\omega^i(du_i\omega^{-j})^{-r_i}A_j(x)^{r_i}\right).$$
This completes the proof.\end{proof}

A polynomial $f(x)\in\mathbb{F}_q[x]$ is called a complete permutation polynomial (CPP) if both $f(x)$ and $f(x)+x$ are permutation polynomials of $\mathbb{F}_q$. By applying the above theorem, we have

\begin{coro}Let the notations be as in Theorem \ref{th4.1}, then
$$f(x)=g(A_0(x))+\sum_{i=0}^{d-1}u_iA_i(x)$$
is a CPP over $\mathbb{F}_{q^d}$ if and only if $\prod_{i=1}^{d-1}u_i(1+du_i\omega^{-i})\ne0$ and both $g(x)$ and $dg(x)+x$ are PPs over $\mathbb{F}_q$.\end{coro}
\begin{proof} We have
$$\xymatrix{
  \mathbb{F}_{q^d} \ar[d]_{A_i(x)} \ar[r]^{f(x)+x}
                & \mathbb{F}_{q^d} \ar[d]^{A_i(x)}  \\
  B_i  \ar[r]_{(1+du_i\omega^{-i})x}
                & B_j          }\quad 1\le i\le d-1 $$
and
$$\xymatrix{
  \mathbb{F}_{q^d} \ar[d]_{A_0(x)} \ar[r]^{f(x)+x}
                & \mathbb{F}_{q^d} \ar[d]^{A_0(x)}  \\
  \mathbb{F}_q  \ar[r]_{x+dg(x)}
                & \mathbb{F}_q          } $$
Hence both $f(x)$ and $f(x)+x$ are PPs over $\mathbb{F}_{q^d}$ if and only if $\prod_{i=1}^{d-1}u_i(1+du_i\omega^{-i})\ne0$ and both $g(x)$ and $dg(x)+x$ are PPs over $\mathbb{F}_q$. We are done.                \end{proof}

Theorem \ref{th4.1} extends Theorem 6.1 in \cite{Yuan22} greatly. We also have the following result, which improves Theorem 6.2 in \cite{Yuan22}.
\begin{thm} Let $n>1$ be a positive integer, $q$ a prime power and $g(x)\in\mathbb{F}_q[x]$. Then the the polynomial
$$f(x)=x^q-x+g(Tr(x)),$$
where $Tr(x)=x+x^q+\cdots+x^{q^{n-1}}$ is a PP over $\mathbb{F}_{q^n}$ if and only if $\gcd(n, q)=1$ and $g(x)$ is a PP over $\mathbb{F}_q$. Moreover, if $f(x)$ is a PP over $\mathbb{F}_{q^n}$ and $g^{-1}(x)$ is the compositional inverse of $g(x)$, then
$$f^{-1}(x)=\frac{1}{n}\left(g^{-1}(Tr(x)/n)-\frac{(n+1)}{2}Tr(x)+\sum_{i=1}^{n}ix^{q^{i}}\right).$$
\end{thm}
\begin{proof} By the assumptions, we have the following diagram
$$\xymatrix{
  \mathbb{F}_{q^n} \ar[d]_{Tr(x)} \ar[r]^{f}
                & \mathbb{F}_{q^n} \ar[d]^{Tr(x)}  \\
  \mathbb{F}_q \ar[r]_{ng(x)}
                & \mathbb{F}_q          } $$
 Hence $f(x)$ is a  PP over $\mathbb{F}_{q^n}$  only if $\gcd(n, q)=1$ and $g(x)$ is a PP over $\mathbb{F}_q$. Moreover we have the following commutative diagrams
$$\xymatrix{ \mathbb{F}_{q^n} \ar[rr]^{f(x)} \ar[dr]_{\varphi_i}
                &  &    \mathbb{F}_{q^n} \ar[dl]^{\psi_i}    \\
                & \mathbb{F}_{q^n}                }\quad i=1, 2,
$$
where $\varphi_1(x)=Tr(x), \psi_1(x)=g^{-1}(Tr(x)/n)$, $\psi_2(x)=x$ and $\varphi_2(x)=f(x)=x^q-x+g(Tr(x))$.

Note that $\varphi_2(x)-g(\varphi_1(x))=x^q-x$ and
$$\sum_{i=1}^ni(x^q-x)^{q^i}=nx^{q^n}-Tr(x)=nx-\varphi_1(x).$$
It follows that
$$x=\frac{1}{n}\left(\varphi_1(x)+\sum_{i=1}^{n}i(\varphi_2(x)-g(\varphi_1(x)))^{q^{i}}\right).$$  Therefore the polynomial $f(x)$ is a PP over $\mathbb{F}_{q^n}$ if and only if $\gcd(n, q)=1$  and $g(x)$ is a PP over $\mathbb{F}_q$. By Theorem \ref{th2.2}, we have
$$f^{-1}(x)=\frac{1}{n}\left(\psi_1(x)+\sum_{i=1}^{n}i(\psi_2(x)-g(\psi_1(x)))^{q^{i}}\right)$$
$$=\frac{1}{n}\left(g^{-1}(Tr(x)/n)+\sum_{i=1}^{n}i(x-Tr(x)/n)^{q^{i}}\right)$$
$$=\frac{1}{n}\left(g^{-1}(Tr(x)/n)-\frac{(n+1)}{2}Tr(x)+\sum_{i=1}^{n}ix^{q^{i}}\right)$$
  This completes the proof.
                \end{proof}

\section{Concluding remarks}
If it is difficult to find  enough  surjections $\varphi_i, \psi_i, 1\le i\le t$ such that the following diagrams commute
$$\xymatrix{ \mathbb{F}_{q} \ar[rr]^{f(x)} \ar[dr]_{\varphi_i}
                &  &    \mathbb{F}_{q} \ar[dl]^{\psi_i}    \\
                & \mathbb{F}_{q}                }\quad i=1, \ldots, t
$$
and there exists a polynomial $F(x_1, \ldots, x_t)\in\mathbb{F}_q[X_1, \ldots, X_t]$ with 
$$F(\varphi_1(x), \ldots, \varphi_t(x))=x.$$
We can use the dual diagrams of the related diagrams
$$\xymatrix{
  A \ar[d]_{\psi} \ar[r]^{f^{-1}}
                & A \ar[d]^{\varphi}  \\
  S_i  \ar[r]_{h_i^{-1}}
                &   S_i        }$$
as we did in \cite{Y22}. In \cite{Y22}, we obtained the compositional inverses of some PPs by using one diagram and its dual diagram. Now we can make use of a new local method  to find the compositional inverses of some specified PPs. We believe, more PPs and their compositional inverses will be found or constructed by the method of this paper. We will continue the work in the further paper.


\begin{thebibliography}{10}


\bibitem{AGW11}A. Akbary, D. Ghioca, and Q. Wang, On constructing permutations
of finite fields, Finite Fields Their Appl., vol. 17, no. 1, pp. 51-67,
Jan. 2011.




\bibitem{CH02} R. S. Coulter and M. Henderson, The compositional inverse of a class
of permutation polynomials over a finite field, Bull. Austral. Math. Soc.,
65(2002), 521-526.


%\bibitem{Ding13} C. Ding, Cyclic codes from some monomials and trinomials, SIAM J. Discrete Math. 27(2013),1977-1994.

\bibitem{DY06} C. Ding and J. Yuan, A family of skew Hadamard difference sets,
J. Comb. Theory, Ser. A 113 (2006),  1526-1535.

%\bibitem{DZ14} C. Ding and Z. Zhou, Binary cyclic codes from explicit polynomials over $GF(2^m)$, Discrete Math. 321(2014), 76-89.


\bibitem{LC07} Y.
Laigle-Chapuy, Permutation polynomials and applications to coding
theory, Finite Fields Appl. 13 (2007),  58--70.


\bibitem{LQW19} K. Li, L. Qu, and Q. Wang, Compositional inverses of permutation
polynomials of the form $x^rh(x^s)$ over finite fields, Cryptogr. Commun.,
11(2019),279-298.



\bibitem{LN97} R. Lidl, H. Niederreiter,
``Finite Fields",  Cambridge University Press, Cambridge, 1997.






\bibitem{NLQW21} T. Niu, K. Li, L. Qu, and Q. Wang,Finding compositional inverses of permutations from the AGW creterion, IEEE Trans. Inf. Theory, 67(2021), 4975-4985.


\bibitem{PL01} Y. H. Park, J. B. Lee, Permutation polynomials and group
permutation polynomials, Bull. Austral. Math. Soc. 63
(2001), 67--74.



\bibitem{TW14} A. Tuxanidy and Q. Wang, On the inverses of some classes of permutations of finite fields, Finite Fields Their Appl., 28(2014),244-281.
\bibitem{TW17} A. Tuxanidy and Q. Wang, Compositional inverses and complete mappings over finite fields, Discrete Appl. Math., 217(2017), 318-329.




\bibitem{W07} Q. Wang, Cyclotomic mapping permutation polynomials over finite
fields, in Sequences, Subsequences, Consequences. Berlin, Germany:
Springer, 2007, pp. 119-128.

\bibitem{W17} Q. Wang, A note on inverses of cyclotomic mapping permutation
polynomials over finite fields, Finite Fields Their Appl., 45(2017),422-427.


\bibitem{Wu14} B. Wu, The compositional inverse of a class of linearized permutation
polynomials over $\mathbb{F}_{2^n}, n$ odd, Finite Fields Their Appl., 29(2014), 34-48.

\bibitem{WL13} B. Wu and Z. Liu, The compositional inverse of a class of bilinear
permutation polynomials over finite fields of characteristic 2, Finite
Fields Their Appl., 24(2013), 136-147.
\bibitem{WL13J} B. Wu and Z. Liu, Linearized polynomials over finite fields revisited,
Finite Fields Their Appl., 22(2013), 79¨C100.

\bibitem{Y22} P. Yuan, Compositional Inverses of AGW-PPs, Adv. Math. Comm., doi:10.3934/amc.2022045.


\bibitem{Yuan22} P. Yuan, Permutation Polynomials and their Compositional Inverses, 2022, arXiv:2206.04252. [Online].

\bibitem{YD11} P. Yuan and C. Ding, Permutation polynomials over finite fields from a
powerful lemma, Finite Fields Their Appl., 17(2011), 560-574.

\bibitem{YD14} P. Yuan, C. Ding, Further results on permutation polynomials over finite fields. Finite Fields Appl. 27 (2014), 88-103.



\bibitem{ZYLHZ19} D. Zheng, M. Yuan, N. Li, L. Hu, and X. Zeng, Constructions of
involutions over finite fields, IEEE Trans. Inf. Theory, 65(2019), 7876-7883.

\bibitem{ZY18} Y. Zheng and Y. Yu, On inverse of permutation polynomials of
small degree over finite fields, II, 2018, arXiv:1812.11812. [Online].


\bibitem{ZWW20} Y. Zheng, Q. Wang, and W. Wei, On inverses of permutation polynomials of small degree over finite fields, IEEE Trans. Inf. Theory,  66(2020), 914-922,.


\bibitem{Z09} M. Zieve, On some permutation polynomials over $\mathbb{F}_q$ of the form
$x^rh(x^{(q-1)/d})$, Proc. Amer. Math. Soc., 137(2009), 2209-2216.










%\bibitem{Wa02} L. Wang, On permutation polynomials, Finite
%Fields Appl. \textbf{8} (2002) 311-322.




\end{thebibliography}
\end{document}